\documentclass{amsart}
\title{Selective strong screenability and a game} 
\author{Liljana Babinkostova and Marion Scheepers} 
\date{\today} 

\usepackage{amsfonts}
\usepackage{amsmath, amsthm}
\usepackage{colonequals}
\usepackage{txfonts}
\usepackage{xcolor}
\newtheorem{theorem}{{\bf Theorem}}
\newtheorem{proposition}[theorem]{{\bf Proposition}}
\newtheorem{lemma}[theorem]{{\bf Lemma}}
\newtheorem{corollary}[theorem]{{\bf Corollary}}
\newtheorem{conjecture}{{\bf Conjecture}}

\newtheorem{problem}{{\bf Problem}}

\newcommand{\naturals}{{\mathbb N}}

\newcommand{\open}{\mathcal{O}}

\makeatletter

\begin{document}
\maketitle
\begin{abstract}
Selective versions of screenability and of strong screenability coincide in a large class of spaces. We show that the corresponding games are not equivalent in even such standard metric spaces as the closed unit interval. We identify sufficient conditions for ONE to have a winning strategy (Theorem \ref{ONEmain}), and necessary conditions for TWO to have a winning strategy in the selective strong screenability game (Theorem \ref{twowinsomega}). 
\end{abstract}

Unless specified otherwise, all topological spaces in this paper are assumed to be infinite. A collection $\mathcal{A}$ of subsets of a topological space $(X,\tau)$ is \emph{discrete} if there is for each $x\in X$ a neighborhood $U$ of $x$ such that $\vert\{A\in\mathcal{A}:\, A\cap U\neq \emptyset\}\vert \le 1$. Note that a finite family of nonempty sets whose closures are disjoint is a discrete family. An infinite family of sets with pairwise disjoint closures need not be discrete, as illustrated by the family $\{\lbrack\frac{1}{2n+1},\;\frac{1}{2n}\rbrack: n\in{\mathbb N}\}$ of disjoint closed subsets of the real line.  A disjoint family of open sets covering a space is automatically a discrete family of open sets.

A family $\mathcal{A}$ of sets \emph{refines} a family $\mathcal{B}$ of sets if there is for each $A\in\mathcal{A}$ a $B\in\mathcal{B}$ such that $A\subseteq B$. The symbol $\mathcal{O}$ denotes the collection of all open covers of the space $(X,\tau)$. When $Y$ is a subset of $X$, then $\mathcal{O}_Y$ denotes the set of covers of $Y$ by sets open in $X$.

R.H. Bing introduced the notions of \emph{screenable} and \emph{strongly screenable} in \cite{Bing}. A topological space $(X,\tau)$ is \emph{strongly screenable} if there is for each open cover $\mathcal{U}$ of $X$ a sequence $(\mathcal{V}_n:n<\omega)$ such that each $\mathcal{V}_n$ is a \emph{discrete} collection of sets, each $\mathcal{V}_n$ refines $\mathcal{U}$, and $\bigcup\{\mathcal{V}_n:n<\omega\}$ is an open cover of $X$. We obtain the notion of being \emph{screenable} by replacing ``{\tt discrete}" in the definition of strong screenability with ``{\tt disjoint}". 

Towards defining the selective version of strong screenability let $\mathcal{A}$ and $\mathcal{B}$ be collections of families of subsets of a set $S$. Assume that the set $S$ is endowed with a topology. Then ${\sf S}_d(\mathcal{A},\mathcal{B})$ denotes the selection principle: 
\begin{quote} For each sequence $(\mathcal{U}_n:n<\omega)$ of elements of $\mathcal{A}$ there is a sequence $(\mathcal{V}_n:n<\omega)$ such that:
\begin{enumerate}
\item{For each $n$, $\mathcal{V}_n$ refines $\mathcal{U}_n$;} 
\item{For each $n$, $\mathcal{V}_n$ is a discrete collection of sets;} 
\item{$\bigcup\{\mathcal{V}_n:n<\omega\}$ is an element of $\mathcal{B}$.} 
\end{enumerate}
\end{quote}
In this notation the property ${\sf S}_d(\mathcal{O},\mathcal{O})$ of a topological space is called \emph{selective strong screenability} of the space. If in (2) of the definition of ${\sf S}_d(\mathcal{A},\mathcal{B})$ we replace {\tt discrete} with {\tt disjoint} we obtain the selection principle $\textsf{S}_c(\mathcal{A},\mathcal{B})$ that was introduced in \cite{Babinkostova}. The corresponding selection principle $\textsf{S}_c(\mathcal{O},\; \mathcal{O})$ for a topological space is the selective version of screenability, called \emph{selective screenability}. Selective screenability was introduced by Addis and Gresham in \cite{AG} under the name \emph{property} \textsf{C}.

Screenability properties are related to several fundamental topological notions, including paracompactness, metrizability and extensions of covering dimension. 
A family $\mathcal{A}$ of sets in a topological space $(X,\tau)$ has the  property of being \emph{locally finite} if there is for each $x\in X$ a neighborhood $U$ of $x$ such that $\vert\{A\in\mathcal{A}:\, A\cap U\neq \emptyset\}\vert$ is finite. A topological space is \emph{paracompact} if for each given open cover there is a locally finite open cover refining the given cover. In \cite{Michael} Michael and, independently, in \cite{Nagami} Nagami proved
\begin{theorem}[Michael, Nagami] A regular space is paracompact if, and only if, it is strongly screenable.
\end{theorem}

Theorem 5 of \cite{Nagami} also proves\footnote{In personal communication Roman Pol and Elzbieta Pol pointed out that Nagami's result can be strengthened to show  that selective screenability and selective strong screenability coincide in normal countably paracompact spaces, and thus in metric spaces.}: 
\begin{theorem}[Nagami]\label{nagami2} A normal, countably paracompact space is screenable if, and only if, it is strongly screenable.
\end{theorem}

The hypothesis of countable paracompactness in Theorem \ref{nagami2} is necessary. To justify this we first comment on the terminology \emph{zero dimensional}: According to Sierpinski \cite{Engelking} a space is \emph{zero-dimensional} if each element has a neighborhood basis consisting of sets that are both open and closed. A space has \emph{covering dimension zero} if each \emph{finite} open cover has a refinement by disjoint open sets, still covering the space. A space is \emph{ultraparacompact} if \emph{each} open cover has a refinement by disjoint open sets still covering the space. Covering dimension zero is also called \emph{strongly zero dimensional}.  

\begin{theorem}[Balogh, \cite{Balogh}]\label{balogh} There is a strongly zerodimensional $\textsf{T}_4$ space that is screenable\footnote{Balogh's space is in fact \emph{selectively} screenable.} but not countably paracompact, and thus not strongly screenable. 
\end{theorem}

In \cite{BKS} it was shown that for regular spaces paracompactness is equivalent to a selective version of paracompactness. Although in these spaces paracompactness is equivalent to strong screenability, (selective) paracompactness does not imply selective screenability: The Hilbert Cube $\lbrack 0,\; 1\rbrack^{{\mathbb N}}$ is compact and metrizable, but is not selectively screenable.

In separable metric spaces selective screenability is related to dimension theory:  If we use $\mathcal{O}_2$ to denote the family of open covers consisting of two sets each, then ${\sf S}_c(\mathcal{O}_2,\mathcal{O})$ corresponds to Alexandroff's notion of \emph{weakly infinite dimensional}. It was an open problem whether Hurewicz's notion of countable dimensionality coincides with Alexandroff's notion of weak infinite dimensionality until R. Pol gave an example of a compact selectively screenable metrizable space that is not countable dimensional \cite{RPol1981}.

In separable metrizable spaces dimension theoretic concepts have been further clarified 
by the study of the \emph{selective screenability game}: Let an ordinal $\alpha>0$ be given. Then ${\sf G}^{\alpha}_c(\mathcal{A},\mathcal{B})$ denotes the following game of length $\alpha$: 
In inning $\gamma<\alpha$ player ONE selects an element $A_{\gamma}$ of $\mathcal{A}$, and TWO then responds with $B_{\gamma}$, a disjoint collection of sets that is a refinement of $A_{\gamma}$. A play $A_0,\, B_0,\, \cdots,\, A_{\gamma},\, B_{\gamma},\,\cdots \hspace{0.1in}\gamma<\alpha$
is won by TWO if $\bigcup\{B_{\gamma}:\gamma<\alpha\}\in\mathcal{B}$; otherwise, ONE wins.
It was proven in \cite{LB2} that a separable metrizable space $X$ is
\begin{enumerate}
\item{of Lebesgue covering dimension $n$ if, and only if, $n$ is minimal such that TWO has a winning strategy in ${\sf G}_c^{n+1}(\mathcal{O},\mathcal{O})$;}
\item{countable dimensional (in the sense of Hurewicz) if, and only if, TWO has a winning strategy in ${\sf G}^{\omega}_c(\mathcal{O},\mathcal{O})$.}
\end{enumerate}
These results inspired the notion of \emph{game dimension}, explored in the papers \cite{LBGD1} and \cite{LBGD2}. 
Even though selective screenability and selective strong screenability are equivalent concepts in normal countably paracompact spaces, the corresponding games have very different characteristics, the topic of this paper. In sections 3 and 4 we report findings regarding player ONE and player TWO, respectively, on the length $\omega$ version of the selective strong screenability game. In section 5 we consider other ordinal lengths for the game. 

\section{The selective strong screenability game}

For ordinal $\alpha>0$ define the game ${\sf G}^{\alpha}_d(\mathcal{A},\mathcal{B})$ as follows: In each inning $\gamma<\alpha$ ONE first selects an $A_{\gamma}$ from $\mathcal{A}$, to which TWO responds with a $B_{\gamma}$ which is a discrete family of sets refining the family $A_{\gamma}$.  A play
\[
  A_0,\, B_0,\, \cdots,\, A_{\gamma},\, B_{\gamma},\,\cdots \hspace{0.1in}\gamma<\alpha
\]
is won by TWO if $\bigcup\{B_{\gamma}:\gamma<\alpha\}\in\mathcal{B}$; otherwise, ONE wins.

Aside from the following easily verified relationships  the games ${\sf G}^{\alpha}_d(\mathcal{A},\mathcal{B})$ and ${\sf G}^{\alpha}_c(\mathcal{A},\mathcal{B})$ are in fact very different from each other:
\begin{itemize}
\item{If TWO has a winning strategy in ${\sf G}^{\alpha}_d(\mathcal{A},\mathcal{B})$, then TWO has a winning strategy in ${\sf G}^{\alpha}_c(\mathcal{A},\mathcal{B})$.}
\item{If ONE has a winning strategy in ${\sf G}^{\alpha}_c(\mathcal{A},\mathcal{B})$, then ONE has a winning strategy in ${\sf G}^{\alpha}_d(\mathcal{A},\mathcal{B})$.}
\end{itemize}

Moreover, certain monotonicity properties hold for this game:
\begin{itemize}
\item{Assume that $\mathcal{A}^{\prime}\supseteq \mathcal{A}$ and $\mathcal{B}^{\prime}\subseteq \mathcal{B}$: If ONE has a winning strategy in the game $\textsf{G}_d^{\alpha}(\mathcal{A},\; \mathcal{B})$ then ONE has a winning strategy in the game $\textsf{G}^{\alpha}_d(\mathcal{A}^{\prime},\mathcal{B}^{\prime})$. If TWO has a winning strategy in the game$\textsf{G}_d^{\alpha}(\mathcal{A}^{\prime},\; \mathcal{B}^{\prime})$ then TWO has a winning strategy in the game $\textsf{G}^{\alpha}_d(\mathcal{A},\mathcal{B})$.}
\item{Let $\alpha<\beta$ be ordinal numbers. If ONE has a winning strategy in the game $\textsf{G}_d^{\beta}(\mathcal{A},\; \mathcal{B})$ then ONE has a winning strategy in the game $\textsf{G}^{\alpha}_d(\mathcal{A},\mathcal{B})$. If TWO has a winning strategy in the game $\textsf{G}_d^{\alpha}(\mathcal{A},\; \mathcal{B})$ then TWO has a winning strategy in the game $\textsf{G}^{\beta}_d(\mathcal{A},\mathcal{B})$.}
\end{itemize}

Also the following fact is easy to verify:
\begin{proposition} Let $(X,\tau)$ be a topological space, let $Y$ be a closed subset of $X$ and let $\alpha>0$ be an ordinal. If ONE has a winning strategy in the game $\textsf{G}^{\alpha}_d(\mathcal{O},\mathcal{O})$ played on $Y$, then ONE has a winning strategy in this game played on $X$. If TWO has a winning strategy in the game $\textsf{G}^{\alpha}_d(\mathcal{O},\mathcal{O})$ played on $X$, then TWO has a winning strategy in this game played on $Y$.
\end{proposition}

\section{Winning strategies for player ONE}

The following version of the Banach-Mazur game on a topological space $(X,\tau)$ with specified subspace $Y$  was defined in \cite{Oxtoby}: 
There is an inning per finite ordinal. In the $n$-th inning ONE chooses a nonempty open subset $O_n$ of $X$ and TWO responds with a nonempty open subset $T_n$ of $X$. The players must obey the rule that for each $n$, $O_n\supseteq T_n\supseteq O_{n+1}$. ONE wins a play 
\[
  O_0,\; T_0,\; O_1,\; T_1,\; \ldots O_n,\; T_n,\; \ldots
\]
if $Y\cap(\bigcap\{O_n:n<\omega\}) \neq \emptyset$. Otherwise, TWO wins the play. 

In \cite{GT}, p. 53, the special case of $Y=X$ of this game is denoted $\textsf{MB}(X)$. We use the notation $\textsf{MB}(Y,X)$ to denote this game in the general case.

\begin{lemma}\label{coverconstruct} If $X$ is a $\textsf{T}_1$-space and $U\neq X$ is an open subset of $X$ such that $\vert U\vert > 1$, then there is an open cover $\mathcal{U}$ of $X$ such that for each $V\in\mathcal{U}$ we have $U\not\subseteq V$. \end{lemma}
\begin{proof} With $U$ and $X$ as given, choose distinct elements $x$ and $y$ in $U$. Then as $X$ is $\textsf{T}_1$ choose open sets $U_x$ and $U_y$, both subsets of $U$, with $x\in U_x\setminus U_y$ and $y\in U_y\setminus U_x$. For any $z\in X\setminus\{x,\; y\}$ choose an open set $U_z\subseteq X\setminus\{x,\; y\}$. Then the open cover $\mathcal{U} = \{U_t:\; t\in X\}$ is as required.
\end{proof}

\begin{lemma}\label{discreteconnected}
A space is connected if, and only if, it is not a union of a discrete collection consisting of more than one nonempty proper subsets. 
\end{lemma}
\begin{proof}
Suppose $X$ is a space and that $\mathcal{F}$ is a collection of nonempty proper subsets of $X$ such that $\mathcal{F}$ is a discrete family, $\vert\mathcal{F}\vert>1$ and $X = \bigcup\mathcal{F}$. Then also $\mathcal{G} = \{\overline{F}:\; F\in\mathcal{F}\}$ is a discrete family of subsets of $X$ that covers $X$, and $\vert\mathcal{G}\vert >1$. Choose $U\in\mathcal{G}$. Then $U$ is nonempty and closed, and as $\mathcal{G}$ is a discrete family, also $V = \bigcup(\mathcal{G}\setminus\{U\})$ is closed. But then $X = U \cup V$ and $U$ and $V$ are disjoint nonempty open sets, whence $X$ is not connected. Conversely, if $X$ is not connected then a family $\{U,\; V\}$ of disjoint nonempty open sets with union $X$ is a discrete collection consisting of more than one nonempty set.
\end{proof}

From now on call a connected set \emph{nontrivial} if it has more than one element. Recall that a family $\mathcal{P}$ of nonempty open subsets of a topological space is said to be a $\pi$-\emph{base} if there is for each nonempty open subset $U$ of the space an element $V$ of $\mathcal{P}$ such that $V\subseteq U$.

\begin{theorem}\label{ONEmain} Let $X$ be a $\textsf{T}_1$ topological space and let $Y$ be a subspace of $X$ such that 
\begin{enumerate}
\item{$X$ has a $\pi$-base consisting of nontrivial connected sets, and}
\item{ONE has a winning strategy in the game $\textsf{MB}(Y,X)$.}
\end{enumerate}
Then ONE has a winning strategy in the game ${\sf G}^{\omega}_d(\mathcal{O},\mathcal{O}_Y)$.
\end{theorem}
\begin{proof}
Let $\sigma$ be ONE's winning strategy in the game $\textsf{MB}(Y,X)$. We may assume that $\sigma$ calls on ONE to play elements of a fixed $\pi$-base consisting of nontrivial connected open sets. Define a strategy $F$ for ONE of the game ${\sf G}^{\omega}_d(\mathcal{O},\mathcal{O})$ as follows:

To begin, consider $O_0 = \sigma(X)$, and apply Lemma \ref{coverconstruct} to define $F(\emptyset)$, ONE's first move in ${\sf G}^{\omega}_d(\mathcal{O},\mathcal{O})$, to be an open cover for which no element contains $O_0$ as a subset. 
If TWO's response is the discrete open refinement $\mathcal{T}_0$, by Lemma \ref{discreteconnected}  the discrete family $\{\overline{T}:T\in\mathcal{T}_0\}$ does not cover $O_0$. Let TWO of the game $\textsf{MB}(Y,X)$ play $T_0 = O_0\setminus\bigcup\{\overline{T}:T\in\mathcal{T}_0\}$ a nonempty open set. 

Let $O_1 = \sigma(T_0)$ be ONE's response in the game $\textsf{MB}(Y,X)$. ONE's move $F(\mathcal{T}_0)$ in the strong screenability game is an open cover of $X$ for which no member has $O_1$ as a subset. 
TWO's response, $\mathcal{T}_1$ is a discrete open refinement of $F(\mathcal{T}_0)$. As $\{\overline{T}:\; T\in\mathcal{T}_1\}$  does not cover $O_1$, $T_1 = O_1\setminus\bigcup\{\overline{T}:T\in\mathcal{T}_1\}$ is a legal move for TWO in the game $\textsf{MB}(Y,X)$. 

In the next inning ONE of the game $\textsf{MB}(Y,X)$ responds with $O_2 = \sigma(T_0,T_1)$. ONE's move $F(\mathcal{T}_0,\; \mathcal{T}_1)$ in the strong screenability game is an open cover of $X$ (as in Lemma \ref{coverconstruct}) for which no member has $O_2$ as a subset.
TWO's response, $\mathcal{T}_2$ is a discrete open refinement of $F(\mathcal{T}_0,\;\mathcal{T}_1)$. By Lemma \ref{discreteconnected} $\{\overline{T}:\; T\in\mathcal{T}_2\}$ cannot cover $O_2$, whence $T_2 = O_2\setminus\bigcup\{\overline{T}:\; T\in\mathcal{T}_2$ is a legal move for TWO of the game $\textsf{MB}(Y,X)$. Then $O_3 = \sigma(T_0,\; T_1,\;T_2\}$ is a legal move for ONE in the Banach-Mazur game, and so on.

This outlines a definition of a strategy $F$ for ONE in the strong screenabilty game. Corresponding to an $F$ play we have a sequence
\[
  O_0 \supseteq T_0\supseteq O_1 \supseteq T_1 \supseteq O_2 \supseteq T_2\supseteq  O_3\supseteq \cdots
\]
of nonempty open sets such that for each $n$ the open set $\bigcup(\mathcal{T}_1\cup\ldots\cup\mathcal{T}_n)$ is disjoint from $O_{n+1}$. Since $\sigma$ is a winning strategy for ONE of the game $\textsf{MB}(Y,X)$, $Y\cap(\bigcap_{n<\infty}O_n)$ is nonempty. Thus $\bigcup_{n<\infty}\mathcal{T}_n$ is not a cover of $Y$, and TWO looses $F$-plays of ${\sf G}^{\omega}_d(\mathcal{O},\mathcal{O}_Y)$.
\end{proof}

\begin{corollary}\label{Peanolike}
If $X$ is a compact locally connected $\textsf{T}_1$-space, then ONE has a winning strategy in the game ${\sf G}^{\omega}_d(\mathcal{O},\mathcal{O})$.
\end{corollary}

Examples of compact locally connected spaces abound. A metrizable compact connected locally connected space is called a \emph{Peano space}. The unit interval is an example of a Peano space. By the Hahn-Mazurkiewicz Theorem a ${\sf T}_2$ space is a Peano space if, and only if, it is a continuous image of the closed unit interval. 

Observe that if $Y$ is a dense $\textsf{G}_{\delta}$ set in the space $X$, then ONE has a winning strategy in $\textsf{MB}(X)$ if, and only if, ONE has a winning strategy in $\textsf{MB}(Y,X)$. 
\begin{corollary}\label{relative2}
Let $Y$ be a dense $\textsf{G}_{\delta}$ subspace of the $\textsf{T}_1$-space $X$ such that
\begin{enumerate}
\item{$X$ has a $\pi$-base consisting of nontrivial connected sets, and}
\item{ONE has a winning strategy in the game on $\textsf{MB}(X)$.}
\end{enumerate}
Then ONE has a winning strategy in the game ${\sf G}^{\omega}_d(\mathcal{O},\mathcal{O}_Y)$ on $X$.
\end{corollary}
${\mathbb P}$, the set of irrational numbers, is a dense $\textsf{G}_{\delta}$ subset of  ${\mathbb R}$, the real line. Corollary \ref{relative2} implies that ONE has a winning strategy in the game $\textsf{G}^{\omega}_d(\mathcal{O},\mathcal{O}_{\mathbb P})$ on the real line. 

\section{Player TWO}

\begin{lemma}\label{ultrapcpt} For a topological space $X$ the following are equivalent:
\begin{enumerate}
\item{$X$ is an ultraparacompact space.}
\item{TWO has a winning strategy in the game ${\sf G}^1_d(\mathcal{O},\mathcal{O})$.}
\end{enumerate}
\end{lemma}

With ${\mathbb S}$ the Sorgenfrey line, ${\mathbb S}\times {\mathbb S}$ is zero-dimensional and regular, but not normal, thus not paracompact, and thus by the Michael-Nagami Theorem, not strongly screenable. Thus, ONE has a winning strategy in the game $\textsf{G}^{\omega}_d(\mathcal{O},\mathcal{O})$ on ${\mathbb S}\times{\mathbb S}$, while TWO has a winning strategy in $\textsf{G}^1_d(\mathcal{O},\mathcal{O})$ on ${\mathbb S}$.  
In \cite{Roy}  P. Roy constructed a complete (non-separable) metric space $X$ of cardinality $2^{\aleph_0}$ which is zero-dimensional, has Lebesgue covering dimension 1, and is not ultraparacompact. Roy's example is a complete zero-dimensional metric space for which TWO does not have a winning strategy in $\textsf{G}^1_d(\mathcal{O},\mathcal{O})$ and thus not in $\textsf{G}_d^{\omega}(\mathcal{O},\mathcal{O})$, as we shall see in Theorem \ref{twowinsomega}.

Zerodimensional Lindel\"of spaces are ultraparacompact. Thus, 
\begin{corollary}\label{zerodim} For Lindel\"of  space $X$ the following are equivalent:
\begin{enumerate}
\item{$X$ is zero-dimensional.}
\item{TWO has a winning strategy in $\textsf{G}^1_d(\mathcal{O},\mathcal{O})$ on $X$.}
\end{enumerate}
\end{corollary}

 Balogh's space mentioned in Theorem \ref{balogh} and constructed in \cite{Balogh} is a union of countably many open sets, each ultraparacompact. Thus TWO has a winning strategy in $\textsf{G}^{\omega}_c(\mathcal{O},\mathcal{O})$. As this space is not strongly screenable ONE has a winning strategy in $\textsf{G}^{\alpha}_d(\mathcal{O},\mathcal{O})$ for each countable ordinal $\alpha$.

The existence of winning strategies for TWO in the relative version of the game seems more delicate. 
The following fact about extending open sets from a subspace to a containing space can be found in Theorem 3 on p. 227 of \cite{KuratowskiI}. Observe that the metric spaces in Lemma \ref{extend} are \emph{not} assumed to be separable.
\begin{lemma}\label{extend} Let $X$ be a metric space and let $Y$ be a subset of $X$. For each family $\{U_i:\; i\in I\}$ of subsets of $Y$ open in the relative topology of $Y$ there exists a family $\{V_i:\; i\in I\}$ of sets open in $X$ such that
\begin{enumerate}
\item{For each $i\in I$ we have $U_i = Y\cap V_i$ and}
\item{For every finite set $J\subseteq I$, if $\bigcap_{j\in J}U_j = \emptyset$, then $\bigcap_{j\in J}\overline{V}_j = \bigcap_{j\in J}\overline{U}_j$, where the closures are computed in $X$.}
\end{enumerate}
\end{lemma}

\begin{lemma}\label{twowinscompactzerodim}
Let $X$ be a metric space and let $Y$ be a closed, ultraparacompact subspace of $X$. 
Then TWO has a winning strategy in the game ${\sf G}_d^{1}(\mathcal{O},\mathcal{O}_Y)$.
\end{lemma}
\begin{proof}
Let an open cover $\mathcal{U}$ of $X$ be given. Since $Y$ is an ultraparacompact space there is in the relative topology of $Y$ a disjoint family $\{U_i:i\in I\}$ of open sets that refines $\mathcal{U}$ and covers $Y$.  
Being disjoint subsets of $Y$ these relatively open sets are in fact closed in $Y$, and thus in $X$ as $Y$ is closed in $X$. By Lemma \ref{extend} we may choose for each $i$ an open subset $V_i$ of $X$ such that $V_i \cap Y = U_i = \overline{U}_i$, such that when $i\neq j$ are elements of $I$, then $\overline{V}_i\cap\overline{V}_j = U_i\cap U_j = \emptyset$, and as each $U_i$ is a subset of an element of the open cover $\mathcal{U}$ of $X$, also each $V_i$ may be taken to be an open subset of that same element of $\mathcal{U}$. But then the refinement $\{V_i:i\in I\}$ of $\mathcal{U}$ is an element of $\mathcal{O}_Y$, and is a discrete family.
\end{proof}

\begin{corollary}\label{sigmacompactzerodimensional} Let $X$ be a metric space and let $Y$ be a subset of a $\sigma$-compact zero-dimensional subset of $X$. Then TWO has a winning strategy in $\textsf{G}_d^{\omega}(\mathcal{O},\mathcal{O}_Y)$.
\end{corollary}
\begin{proof}
Let $Y\subseteq C\subseteq X$ be given with $C$ zerodimensional and $\sigma$-compact. Write $C = \bigcup_{n<\omega}C_n$ where each $C_n$ is compact. By Lemma \ref{twowinscompactzerodim} fix for each $n$ a winning strategy $\sigma_n$ of TWO in the game $\textsf{G}^1_d(\mathcal{O},\mathcal{O}_{C_n})$. Then the strategy of responding to ONE's move in inning $n$ using the strategy $\sigma_n$ is winning for TWO in $\textsf{G}_d^{\omega}(\mathcal{O},\mathcal{O}_Y)$.
\end{proof}
The example after Theorem \ref{twowinsomega2} shows that game-length $\omega$ in Corollary \ref{sigmacompactzerodimensional} is optimal.
\begin{theorem}\label{twowinsomega}
Let $X$ be a metrizable space and let $Y$ be a subspace of $X$.
If TWO has a winning strategy in the game ${\sf G}_d^{\omega}(\mathcal{O},\mathcal{O}_Y)$ on $X$, then  $Y$ is a subset of a union of countably many closed, strongly zero-dimensional subsets of $X$.
\end{theorem}
\begin{proof}
Let $F$ be a winning strategy for TWO in the game ${\sf G}^{\omega}_d(\mathcal{O},\mathcal{O}_Y)$. 
Let $d$ be a compatible metric for the topology of $X$, and for each positive integer $n$ let $\mathcal{B}_n$ be the set
\[
\{U\subset X: U\mbox{ open and }{\sf diam}_d(U)<\frac{1}{2^n} \}.
\]
Define $C_{\emptyset}:= \bigcap\{\overline{\bigcup F(\mathcal{B}_n)}: 0<n<\omega\}$. And for each sequence $(n_1,\cdots,n_k)$ of positive integers, define
           $C_{n_1,\cdots,n_k} := \bigcap\{\overline{\bigcup F(\mathcal{B}_{n_1},\cdots,\mathcal{B}_{n_k}, \mathcal{B}_m)}: 0< m<\omega\}$.
           
We claim:
\begin{itemize}
\item[(a)]{Each $C_{n_1,\cdots,n_k}$, as well as $C_{\emptyset}$, is a closed, strongly zerodimensional set.} 
\item[(b)]{$Y \subseteq \bigcup\{C_{\tau}: \, \tau\in^{<\omega}\omega\}$.}
\end{itemize}           
Towards proving (a): 
 Let $A$ and $B$ be disjoint nonempty closed subsets of $C = C_{n_1,\cdots,n_k}$. As $C$ is closed in $X$, so are $A$ and $B$. Since $A$ and $B$ are disjoint, fix $\epsilon>0$ so that for any $x\in A$ and any $y\in B$ we have $d(x,y)>2\epsilon$. Then fix $m$ large enough that $\frac{1}{2^m}<\epsilon$. Then the discrete (in $X$) family $\{C\cap \overline{U}: U\in F(\mathcal{B}_{n_1},\cdots,\mathcal{B}_{n_k},\mathcal{B}_m)\}$ is an open (in the relative topology of $C$) cover of $C$. Moreover, the family 
 $\mathcal{U} =  \{C\cap \overline{U}: U\cap A\neq \emptyset \mbox{ and } U \in F(\mathcal{B}_{n_1},\cdots,\mathcal{B}_{n_k},\mathcal{B}_m)\}$
 is a discrete family of clopen sets in $C$, whence $U =\bigcup\mathcal{U}$ is clopen in $C$. $U$ contains $A$ and by the choice of $\epsilon$ and $m$ is disjoint from $B$.

Towards proving (b), suppose that on the contrary $x\in Y\setminus (\bigcup\{C_{\tau}: \, \tau\in^{<\omega}\omega\})$. As $x$ is not an element of $C_{\emptyset}$, choose an $n_1$ such that $x$ is not in $\overline{\bigcup F(\mathcal{B}_{n_1})}$. Then as $x$ is not an element of $C_{n_1}$, choose an $n_2$ such that $x$ is not in $\overline{\bigcup F(\mathcal{B}_{n_1}, \mathcal{B}_{n_2})}$, and so on. In this way we obtain an $F$-play of the game ${\sf G}^{\omega}_d(\mathcal{O},\mathcal{O}_Y)$ in which TWO lost since TWO did not cover $x\in Y$. This contradicts the hypothesis that $F$ is a winning strategy for TWO.
\end{proof}

\begin{corollary}\label{metrizable}
If $X$ is a metrizable space, then the following are equivalent:
\begin{enumerate}
\item{TWO has a winning strategy in $\textsf{G}_d^{\omega}(\mathcal{O},\mathcal{O})$.}
\item{$X$ is ultraparacompact.}
\item{TWO has a winning strategy in $\textsf{G}^1_d(\mathcal{O},\mathcal{O})$.}
\end{enumerate}
\end{corollary}
\begin{proof}
{\flushleft{\underline{$(1) \Rightarrow (2)$: }}} By Theorem \ref{twowinsomega}, $X$ is a union of countably many closed sets, each strongly zerodimensional. By the countable sum theorem - see \cite{Nagata} Theorem II.2 A) - $X$ is strongly zerodimensional. As $X$ is metrizable the Katetov-Morita Theorem - see Theorem II.7 of \cite{Nagata} - $X$ has covering dimension zero. Thus, by Proposition 3.2.2 of \cite{Engelking}, $X$ is ultraparacompact.
{\flushleft{\underline{$(2) \Rightarrow (3)$: }}} This implication is Lemma \ref{twowinscompactzerodim} since $X$ is metrizable.
{\flushleft{\underline{$(3) \Rightarrow (1)$: }}} This is left to the reader.
\end{proof}

\begin{corollary}\label{reals} Let $Y$ be a subspace of the real line ${\mathbb R}$. Then the following are equivalent:
\begin{enumerate}
\item{TWO has a winning strategy in $\textsf{G}^{\omega}_d(\mathcal{O},\mathcal{O}_Y)$.}
\item{$Y$ is a first category set of real numbers.}
\item{TWO has a winning strategy in the game $\textsf{MB}(Y,{\mathbb R})$.}
\end{enumerate}
\end{corollary}
\begin{proof}
{\flushleft{\underline{(1)$\Rightarrow$(2): }}} Observe that a closed, zerodimensional set of real numbers is nowhere dense. Apply Theorem \ref{twowinsomega}.

{\flushleft{\underline{(2)$\Rightarrow$(1): }}} As $Y$ is a first category set of real numbers it is a subset of a union of countably many closed, nowhere dense sets.  ${\mathbb R}$ is $\sigma$-compact, whence $Y$ is a subset of a union of countably many compact zerodimensional subsets of ${\mathbb{R}}$. By the Hurewicz-Tumarkin Theorem $Y$ is a subset of a $\sigma$-compact zero-dimensional subset of ${\mathbb R}$. Apply Corollary \ref{sigmacompactzerodimensional}.  

{\flushleft{\underline{(2)$\Leftrightarrow$(3): }}} This is a direct application of Theorem 1 of \cite{Oxtoby}.
\end{proof}

In \cite{Kulesza} Kulesza constructs a complete, zerodimensional metric space $K$ that is not ultraparacompact. Indeed, $K$  has covering dimension 1. On p. 111 of \cite{Kulesza} $K$ is represented as $K = P_1 \cup \bigcup_{m\in\naturals}P^m_2$ where the subspace $P_1$ is homeomorphic to $D(\aleph_1)^{\omega}$ and each $P^m_2$ is, by \cite{Kulesza} Lemmas 3.3 and 3.4 and the remarks on \cite{Kulesza}, p. 113, a strongly zerodimensional closed (and nowhere dense) subset of the space $K$.

\begin{corollary}\label{Kuleszaspace} On the space $K$
TWO does not have a winning strategy in $\textsf{G}^{\omega}_d(\mathcal{O},\mathcal{O}_{P_1})$.
\end{corollary}
\begin{proof}
Suppose that, on the contrary, TWO has a winning strategy. By Theorem \ref{twowinsomega} $P_1$ is contained in a union of countably many closed, strongly zerodimensional subsets of $K$. But also each of the subspaces $P^m_2$ is a closed, strongly zerodimensional subset of $K$. Thus, $K$ is the union of countable many closed, strongly zerodimensional subsets. By Theorem 4.1.9 in \cite{Engelking} $K$ has covering dimension 0, contradicting the fact that $K$ has covering dimension larger than $0$.
\end{proof}

Incidentally, note that the argument in the proof of Theorem \ref{twowinsomega} also
gives:
\begin{theorem}\label{twowinsomega2}
Let $X$ be a  metric space and let $Y$ be a subspace of $X$.
If TWO has a winning strategy in ${\sf G}^{1}_d(\mathcal{O},\mathcal{O}_Y)$, then $Y$ is a subset of a closed, strongly zerodimensional subset of $X$.
\end{theorem}
\begin{proof}
In the argument in the proof of Theorem \ref{twowinsomega} we see that $Y\subseteq C_{\emptyset}$.
\end{proof}

Thus, for example, TWO has a winning strategy in the game $\textsf{G}^{\omega}_d(\mathcal{O},\mathcal{O}_{\mathbb Q})$, but does not have a winning strategy in the game $\textsf{G}^{1}_d(\mathcal{O},\mathcal{O}_{\mathbb Q})$.

\section{Longer games}

For any space $(X,\tau)$ there is an ordinal $\alpha \le \vert X\vert$ such that TWO has a winning strategy in the game ${\sf G}^{\alpha}_d(\mathcal{O},\mathcal{O})$ on $X$. Thus, we may define for the space
\[
   \textsf{tp}_d(X,\tau) = \min\{\alpha>0: \mbox{ TWO has a winning strategy in the game }\textsf{G}^{\alpha}_d(\mathcal{O},\mathcal{O})\}.
\]
Since every separable metric space is a union of at most $\aleph_1$ zerodimensional subsets we find that for each separable metrizable space $(X,\tau)$, ${\sf tp}_d(X,\tau)\le \omega_1$.

Let $\alpha$ be an infinite ordinal with Cantor normal form $\alpha = \omega^{\beta_1}\cdot n_1 + \cdots + \omega^{\beta_m}\cdot n_m + n_{m+1}$ where $\beta_1 > \cdots > \beta_m > 0$ and $n_{i}<\omega$ for each $i\le n+1$. Define $\alpha^{-}$ as follows:

\[
   \alpha^{-} = \left\{\begin{tabular}{ll}
                                 $\alpha$ & if $n_{m+1}=0$ and $\beta_m>1$\\
                                 $\omega^{\beta_1}\cdot n_1 + \cdots + \omega^{\beta_m}\cdot (n_m-1)+1 $ & if $n_{m+1}=0$ and $\beta_m=1$\\
$\omega^{\beta_1}\cdot n_1 + \cdots + \omega^{\beta_m}\cdot n_m +1$ & otherwise
                                 \end{tabular}
                       \right.
\]

\begin{corollary}\label{successorordinals} Let $X$ be a metrizable space and let $\alpha$ be an infinite countable ordinal. If TWO has a winning strategy in $\textsf{G}^{\alpha}_d(\mathcal{O},\mathcal{O})$ on $X$ then TWO has a winning strategy in $\textsf{G}^{\alpha^{-}}_d(\mathcal{O},\mathcal{O})$ on $X$.
\end{corollary}
\begin{proof}
For consider a winning strategy $\sigma$ of TWO. We need only consider ordinals $\alpha$ for which $\alpha>\alpha^{-}$.
{\flushleft{\underline{Case 1: $n_{m+1} = 0$.}  }} We may assume that $\beta_m = 1$. After $\omega^{\beta_1}\cdot n_1 + \cdots + \omega^{\beta_m}\cdot (n_m-1)$ innings TWO has covered a part, $U$, of the space $X$, and a closed set $C = X\setminus U$ remains to be covered. Using $\sigma$ TWO has a winning strategy in the game $\textsf{G}^{\omega}_d(\mathcal{O},\mathcal{O})$ on $C$. Now Theorem \ref{twowinsomega2} implies that the closed set $C$ is strongly zero-dimensional. Since $X$ is metrizable, $C$ is ultraparacompact. Thus, TWO has a winning strategy that wins $\textsf{G}^{\alpha^{-}}_d(\mathcal{O},\mathcal{O})$ on $X$.

{\flushleft{\underline{Case 2: $n_{m+1} > 0$.}  }} We may assume that $n_{m+1}>1$. After $\omega^{\beta_1}\cdot n_1 + \cdots + \omega^{\beta_m}\cdot n_m$ innings TWO has covered a part, $U$, of the space $X$, and a closed set $C = X\setminus U$ remains to be covered. Using $\sigma$ TWO has a winning strategy in the game $\textsf{G}^{n_{m+1}}_d(\mathcal{O},\mathcal{O})$ on $C$. Now Theorem \ref{twowinsomega2} implies that the closed set $C$ is strongly zero-dimensional. As $X$ is  metrizable, $C$ is ultraparacompact. Thus, TWO has a winning strategy that wins $\textsf{G}^{\alpha^{-}}_d(\mathcal{O},\mathcal{O})$ on $X$.
\end{proof}

Since the unit interval is a Peano space, Corollary \ref{Peanolike} implies that ONE has a winning strategy in the game ${\sf G}^{\omega}_d(\open,\open)$. We show that TWO has a winning strategy in ${\sf G}_d^{\omega+1}(\open,\open)$ on the unit interval. The key to the argument is Lebesgue's covering lemma:
\begin{theorem}[Lebesgue]\label{lebesgue} If $(X,d)$ is a compact metric space then there is for each open cover $\mathcal{U}$ of $X$ a positive real number $\delta$ such that for each set $Y\subset X$ for which the $d$-diameter is less than $\delta$ there is a set $U\in\mathcal{U}$ such that $Y\subseteq U$.
\end{theorem}

\begin{lemma}\label{intervals} Let $\lbrack a,\, b\rbrack$ be an open interval of positive length $L$. Let $\mathcal{U}$ be a cover of $\lbrack a,\, b\rbrack$ by sets open in $\lbrack 0,\, 1\rbrack$. Then there is a finite discrete open refinement $\mathcal{V}$ of $\mathcal{U}$ such that $\bigcup\mathcal{V}\subset \lbrack a,\, b\rbrack$ and $\lbrack a,\, b\rbrack\setminus \bigcup\mathcal{V}$ is a union of finitely many disjoint closed intervals whose lengths add up to at most $\frac{L}{2}$.
\end{lemma}
\begin{proof}
Using the Lebesgue covering lemma and the compactness of $\lbrack a,\, b\rbrack$, choose a positive real number $\delta$ as in Theorem \ref{lebesgue}. Then choose $\epsilon<\delta$ so that $M:=\frac{L}{\epsilon}$ is an even integer. Choosing $a_0 = a$ and $a_{i+1} = a_i + \epsilon$ for $i<M$ we find that each of the intervals $\lbrack a_i,\, a_{i+1}\rbrack$, $0\le i<M$ is a subset 
of an element of $\mathcal{U}$. Put $\mathcal{V} = \{\left( a_i,\, a_{i+1}\right):\, i<M \mbox{ odd}\}$. Then $\mathcal{V}$ is as required.
\end{proof}

\begin{theorem}\label{twounitinterval}
 TWO has a winning strategy in ${\sf G}_d^{\omega+1}(\open,\open)$ on the closed unit interval.
\end{theorem}
\begin{proof}
Player TWO's strategy in ${\sf G}^{\omega+1}_d(\open,\open)$ is as follows:
In the first inning player TWO applies Lemma \ref{intervals} to the open cover $O_1$ of $\lbrack 0,\, 1\rbrack$ played by ONE to obtain the open refinement $\mathcal{V}_1$ for which $\lbrack 0,\, 1\rbrack\setminus \bigcup\mathcal{V}_1$ is a union of finitely many closed disjoint intervals, $I^1_1, \cdots, I^1_{n_1}$ with lengths adding up to at most $\frac{1}{2}$.

When ONE plays the open cover $O_2$ next, TWO applies Lemma \ref{intervals} to each $I^1_j$ to find a discrete open refinement $\mathcal{V}_{2,j}$ of $O_2$ with all elements subsets of $I^1_j$, and with $I^1_j\setminus \bigcup\mathcal{V}_{2,j}$ a union of finitely many disjoint closed subintervals of $I^1_j$ of positive length with lengths adding up to at most $\frac{length(I^1_j)}{2}$, and then TWO responds with $\mathcal{V}_2 = \cup_{j\le n_1}\mathcal{V}_{2,j}$. It follows that $\lbrack 0,\, 1\rbrack \setminus (\bigcup\mathcal{V}_1 \cup \bigcup{\mathcal{V}_2})$ is a union of finitely many closed, disjoint, intervals of positive length $I^2_1, \cdots, I^2_{n_2}$ with length adding up to at most $\frac{1}{4}$.

By applying this strategy to the next open covers chosen by ONE, we find that after countably many moves the set $\lbrack 0,\, 1\rbrack \setminus \bigcup\cup_{j=1}^{\infty}\mathcal{V}_j$ is compact and zero dimensional. Then by Lemma \ref{twowinscompactzerodim} TWO wins in one more inning.
\end{proof}

\section{Remarks and Questions}

Also for relative versions of the selective strong screenability game one could define the corresponding length ordinals:
For a subspace $Y$ of a topological space $(X,\tau)$, define
\[
  \textsf{tp}_d(X,Y,\tau) = \min\{\alpha \in \textsf{ON}: \mbox{ TWO has a winning strategy in the game }\textsf{G}^{\alpha}_d(\mathcal{O},\mathcal{O}_Y)\}.
\]
Thus, $\textsf{tp}_d(X,\tau) = \textsf{tp}_d(X,X,\tau)$.

\begin{problem}
Is there a topological space $X$ and a subspace $Y$ for which $\textsf{tp}_d(X,Y,\tau) = 2$?
\end{problem}

\begin{problem}
Is there a topological space $X$ for which $\textsf{tp}_d(X,\tau) = 2$?
\end{problem}

There are complete metric spaces that are zero-dimensional but not ultraparacompact. See for example \cite{Kulesza} and \cite{Roy}. In these spaces TWO does not have a winning strategy in the game $\textsf{G}^1_d(\mathcal{O},\mathcal{O})$. It is not clear whether more can be proven:
\begin{problem}\label{ultrapcptmetric}
If $X$ is a complete metric space that is not ultraparacompact, does ONE have a winning strategy in the game $\textsf{G}^{\omega}_d(\mathcal{O},\mathcal{O})$ on $X$?
\end{problem} 

In connection with  Theorem \ref{ONEmain}, it would be interesting to know:
\begin{problem} Let $Y$ be a set of real numbers. Are the following statements equivalent?
\begin{enumerate}
\item{ONE has  a winning strategy in the game $\textsf{MB}(Y,{\mathbb R})$.}
\item{ONE has a winning strategy in the game $\textsf{G}^{\omega}_d(\mathcal{O},\mathcal{O}_Y)$}
\end{enumerate}
\end{problem}

Our results on the closed unit interval and some heuristic arguments suggest:
\begin{conjecture}\label{unitcubes}
For each positive integer $n$ ONE has a winning strategy in ${\sf G}_d^{\omega\cdot n}(\open,\open)$, and TWO has a winning strategy in ${\sf G}_d^{\omega\cdot n+1}(\open,\open)$ on $\lbrack 0,\, 1\rbrack^n$.
\end{conjecture}

\section{Acknowledgements}

We thank Roman Pol and Rodrigo Dias for very informative communications that drastically improved the contents of this paper.


\begin{thebibliography}{}

\bibitem{AG} D. F. Addis and J.H. Gresham, \emph{A class of infinite-dimensional spaces. Part I: Dimension theory and Alexandroff ’s problem}, {\bf Fundamenta Mathematicae} 101:3 (1978),  195–205.

\bibitem{Babinkostova} L. Babinkostova,  \emph{Selection Principles in Topology} (Macedonian), Ph.D. thesis. {\bf University of St. Cyril and Methodius}, Macedonia, 2001.

\bibitem{LB2} L. Babinkostova, \emph{Selective screenability game and covreing dimension}, {\bf Topology Proceedings} 29:1 (2005), 13 - 17

\bibitem{LBGD1} L. Babinkostova, \emph{Topological games and covering dimension}, {\bf Topology Proceedings} 38 (2011), 99-120.

\bibitem{LBGD2} L. Babinkostova, \emph{Topological groups and covering dimension}, {\bf Topology and its Applications} 158:12 (2011), 1460-1470. 

\bibitem{BKS} L. Babinkostova, Lj.D.R. Ko\v{c}inac and M. Scheepers, \emph{Notes on selection principles in topology (I): Paracompactness}, {\bf Journal of the Korean Mathematical Society} 42:4 (2005), 709 - 721.

\bibitem{Balogh} Z. Balogh, \emph{A normal, screenable, nonparacompact space in ZFC}, {\bf Proceedings of the American Mathematical Society} 126:6 (1998), 1835 - 1844.

\bibitem{Bing} R.H. Bing, \emph{Metrization of topological spaces}, {\bf Canadian Journal of Mathematics} 3 (1951), 175 - 186.

\bibitem{Engelking} R. Engelking, \emph{Dimension Theory}, {\bf North-Holland Publishing Company} (1978).

\bibitem{GT} F. Galvin and R. Telgarsky, \emph{Stationary strategies in topological games}, {\bf Topology and its Applications} 22 (1986), 51 - 69.

\bibitem{Kulesza} J. Kulesza, \emph{An example in the dimension theory of metrizable spaces}, {\bf Topology and its Applications} 35 (1990), 109 - 120.

\bibitem{KuratowskiI} C. Kuratowski, \emph{Topology} Vol. I, {\bf Academic Press} 1966.

\bibitem{Michael} E. Michael, \emph{A note on paracompact spaces}, {\bf Proceedings of the American Mathematical Society} 4 (1953), 831 - 838 

\bibitem{Nagami} K. Nagami, \emph{Paracompactness and Strong Screenability}, {\bf Nagoya Mathematics Journal}, 8 (1955), 83 - 88.

\bibitem{Nagata} J. Nagata, \emph{Modern Dimension Theory}, {\bf Biblioteca Mathematica} (VI), {\bf John Wiley \& Sons Inc}, 1965. 

\bibitem{Oxtoby} J.C. Oxtoby, \emph{The Banach-Mazur game and Banach category theorem}, in {\bf Contributions to the Theory of Games, Volume 3}, Princeton University Press (1957), 159 - 164. 

\bibitem{RPol1981} R. Pol, \emph{A weakly infinite-dimensional compactum which is not countable-dimensional}, {\bf Proceedings of the American Mathematical Society} 82:4 (1981),  634 – 636.

\bibitem{Roy} P. Roy, \emph{Nonequality of dimensions for metric spaces}, {\bf Transactions of the American Mathematical Society} 134:1 (1968), 117 -  132.

\end{thebibliography}
\end{document}